\documentclass[10pt]{article}
\usepackage{amsthm}
\usepackage{amssymb,amsfonts,amsmath,amsopn,amstext,amscd,color,latexsym,mathrsfs, stmaryrd}
\usepackage[shortlabels]{enumitem}
\usepackage{mathabx, colonequals, comment}
\usepackage[utf8]{inputenc}
\usepackage[all, cmtip]{xy}
\usepackage{mathtools}
\usepackage{enumitem}
\usepackage{tikz}
\usepackage{appendix}
\usepackage{textcase}
\usetikzlibrary{cd}
\usetikzlibrary{decorations.pathmorphing,arrows,positioning}

\usepackage{microtype}
\usepackage{subfiles}

\usepackage[colorlinks, pagebackref]{hyperref}
\hypersetup{
colorlinks=true,
citecolor=red,
linkcolor=blue,
urlcolor=cyan}

\theoremstyle{plain}

\newtheorem{thm}{\bf Theorem}[subsection]
\newtheorem{lem}[thm]{\bf Lemma}
\newtheorem{cor}[thm]{\bf Corollary}
\newtheorem{prop}[thm]{\bf Proposition}

\theoremstyle{definition}
\newtheorem{nota}[thm]{\bf Notations}

\newtheorem{rem}[thm]{\bf Remark}
\newtheorem{example}[thm]{\bf Example}

\theoremstyle{definition}
\newtheorem{defn}[thm]{\bf Definition}

\usepackage{chngcntr}

\newcommand{\bbF}{{\mathbb F}}

\newcommand{\bbQ}{{\mathbb Q}}

\newcommand{\cT}{{\mathcal T}}

\newcommand{\sH}{{\mathscr H}}

\newcommand{\sR}{{\mathscr R}}

\newcommand{\Frob}{{\rm Frob}}

\newcommand{\lsub}[2]{\prescript{}{#1}{#2}}

\title{Generic vanishing on homogeneous spaces in arbitrary characteristic}

\author{
  Ankit Rai\\
  \small Department of Mathematics, University at Buffalo, Buffalo, NY 14260\\
  \small \texttt{arai22@buffalo.edu}
  \and
  K.\,V. Shuddhodan\\
  \small Department of Mathematics, University of Notre Dame, Notre Dame, IN 46556\\
  \small \texttt{skadattu@nd.edu}
}

\date{} 

\begin{document}
\maketitle

\begin{abstract}

Let $X$ be a proper homogeneous space for a connected algebraic group
$G$ over an algebraically closed field. For locally closed smooth affine subvarieties
$W,Z\subset X$, we show that 
\[
(-1)^{\dim X-\dim W+\dim Z}\chi(gW\cap Z)\geq 0
\]
for generic $g\in G$. This extends the characteristic-zero theorem of Sch\"urmann--Simpson--Wang.
Over finite fields, our methods give a trace-function identity
on a dense open subset of $G$ and a Lang--Weil estimate for the
non-generic locus.

\end{abstract}


\numberwithin{equation}{section}

\section{Introduction}
\label{sec:introduction}

Let $X$ be a projective homogeneous space for a connected algebraic
group $G$, and let $W,Z\subset X$ be locally closed smooth affine
subvarieties. For a generic translate $gW$, Sch\"urmann--Simpson--Wang
\cite{schurmann2025new} proved in characteristic zero that
\[
(-1)^{\dim(gW\cap Z)}\chi(gW\cap Z)\geq 0.
\]
They obtain this sign statement from a generic vanishing theorem for
the family of intersections \(gW\cap Z\). Such statements also occur
in the work of Beilinson \cite{Bei87} and Katz
\cite[Corollary A.5]{Kat93}. We prove the corresponding generic
vanishing theorem over an arbitrary algebraically closed field.

Our result fits a longer tradition. Green--Lazarsfeld
\cite{green1987deformation} treated line bundles on abelian varieties.
Subsequent work extended generic vanishing to \(\ell\)-adic sheaves on
tori (Gabber--Loeser \cite{gabber1996faisceaux}), to \(D\)-modules on
abelian varieties (Schnell \cite{schnell2015holonomic}), to
constructible sheaves on abelian varieties (Weissauer
\cite{MR3498921}; Bhatt--Schnell--Scholze \cite{bhatt2018vanishing}),
and to commutative algebraic groups (Forey--Fres\'an--Kowalski
\cite{forey2021arithmetic}). In all of these the parameter is a
family of \emph{characters} or \emph{local systems} on a fixed
support. Here the sheaf is fixed and the support varies by
translation.

The point is that the relevant family is not merely a family of
intersections; it is a Radon correspondence over the group \(G\).
After a smooth pullback along an orbit map, this correspondence
becomes a product, and the base-change statement reduces to the
K\"unneth formula, which is unconditional in \(D^b_c\) with finite
tor-dimension (the convention adopted in
\S\ref{sec:notations-preliminaries}). From this one obtains
\(t\)-exactness of the Radon transform, and the vanishing theorem
follows from the support bounds for perverse sheaves on \(G\).

We state the result for locally closed affine subvarieties;
Section~\ref{sec:affine=smooth-correspondence} proves the more
flexible version in which the inclusions are replaced by quasi-finite
affine morphisms.

\subsection{Summary of main results of the article}
\label{sec:summary-main-results-article}

Let $X$ be a projective variety over an algebraically closed field $k$,
equipped with an action of a connected algebraic group $G$. We assume
that, for some (and hence every) closed point $x\in X$, the orbit map
\(\mu_x:G\to X\) is smooth and surjective. Let
\(\phi:W\hookrightarrow X\) and \(f:Z\hookrightarrow X\) be locally
closed affine immersions, with \(W\) smooth over \(k\). Under these
assumptions one has a Radon transform

\[
\sR_{W!} : \mathrm{D}^{b}_{c}(X) \rightarrow \mathrm{D}^{b}_{c}(G).
\]

Theorem~\ref{thm:radon-t-exact} proves that
\(\sR_{W!}\circ f_{*}\) is \(t\)-exact for the perverse
\(t\)-structure. We record the resulting support-locus statement.

\begin{nota}
\label{nota:intro-main-thm}
For \(g\in G(k)\), let
\({}_{g}\phi^Z:gW\cap Z\to Z\) and
\({}_{g}f^Z:gW\cap Z\to gW\) be the induced maps. For
\(K\in D^b_c(Z)\), set
\[
G_{i}(K):=\{g \in G(k)\mid
\mathrm{H}_{c}^{i}(gW,{}_{g}f^Z_{*}{}_{g}\phi^{Z*}K
[\dim W-\dim X]) \neq 0\}.
\]
\end{nota}

\begin{thm}
\label{thm:intro-main-thm}(see Corollary \ref{cor:generic-vanishing-general})
Let \(K\) be a perverse sheaf on \(Z\). There is a dense open subset
\(U\subset G\) such that
\begin{enumerate}[(i)]
\item \(G_i(K)\cap U=\emptyset\) for \(i<0\);
\item for \(i\geq 0\),
\[
\dim(U\cap G_{i}(K)) \leq \dim G-i.
\]
\end{enumerate}
Moreover, after replacing \(U\) by a smaller dense open subset, for
all \(g\in U(k)\) and all \(i\neq \dim W-\dim X\) one has
\[
\mathrm{H}^{i}_{c}(gW,{}_{g}f^Z_{*}{}_{g}\phi^{Z*}K)
\simeq
\mathrm{H}^{i}(Z,{}_{g}\phi^Z_{!}{}_{g}\phi^{Z*}K)
=0.
\]
\end{thm}

Together with Laumon's comparison of the ordinary and compactly-supported
Euler characteristics \cite{Lau81}, this gives the sign statement.

\begin{cor}
\label{cor:intro-sign-euler}(see Corollary \ref{cor:positivity-euler-char})
Assume that the coefficients form a field, and let \(K\) be a perverse
sheaf on \(Z\). There is a dense open subset \(U\subset G\) such that,
for all \(g\in U(k)\),
\[
(-1)^{\dim X-\dim W}\chi(gW \cap_{f,\phi} Z,{}_{g}\phi^{Z*}K) \geq 0.
\]
\end{cor}

If \(Z\) is smooth and \(K=\Lambda_Z[\dim Z]\), this is the ordinary
Euler-characteristic inequality
\[
(-1)^{\dim X-\dim W+\dim Z}\chi(gW\cap Z)\geq 0.
\]

This includes Kumar's sign conjecture for triple intersections of
Schubert cells \cite[Conjecture D]{kumar2024conjectural} (conjectured
independently by Knutson--Zinn-Justin and Mihalcea; see also
\cite[Conjecture 9.4]{fan2022pieri}). In characteristic zero the
conjecture was proved in some cases by Knutson--Zinn-Justin
\cite[Theorem 4]{knutson2022schubert} and in full generality by
Sch\"urmann--Simpson--Wang \cite[Corollary 1.5]{schurmann2025new};
Corollary~\ref{cor:intro-sign-euler} extends the result to arbitrary
characteristic.

When \(k=\overline{\bbF_q}\) and the data are defined over \(\bbF_q\),
the same \(t\)-exactness theorem has arithmetic consequences. For a
perverse Weil sheaf \(K\), the trace function of
\(\sR_{W!}\circ f_*K\) is, on a dense open subset of \(G\), the
Frobenius trace on the single middle group appearing in
Theorem~\ref{thm:intro-main-thm}. The complement of this open has
\(\bbF_{q^n}\)-points \(O(q^{n(\dim G-1)})\), by Lang--Weil. Under
a middle-extension (or no-boundary-contribution) condition at
infinity, the same trace identity gives a Weil bound for the resulting `exponential sums'
on the intersections \(gW\cap_{f,\phi}Z\).
Section~\ref{sec:arithmetic-corollaries} records the precise
statements.

\section{Notation and preliminaries}
\label{sec:notations-preliminaries}


Throughout, \(k\) is an algebraically closed field and \(\ell\) is a
prime invertible in \(k\). The coefficient ring \(\Lambda\) is either a
finite ring of torsion prime to \(\mathrm{char}(k)\), or an algebraic
extension \(E/\bbQ_{\ell}\). All schemes are separated and of finite
type over \(k\), and all morphisms are \(k\)-morphisms.

For a scheme \(X\), write \(D^b_c(X)\) for the bounded derived category
of constructible \'{e}tale sheaves of finite tor-dimension with
coefficients in \(\Lambda\). A \emph{sheaf} means an object of
\(D^b_c(X)\). This category carries the standard \(t\)-structure
\cite[1.1.2, (e)]{Del80} and the middle perverse \(t\)-structure
\cite[\S 4.0]{BBDG18}. We denote standard cohomology sheaves by
\(\mathcal{H}^i\).

\subsection{\'Etale local triviality of a correspondence between $G$ and $X$}
\label{sec:group-action}

Let \(G\) be an algebraic group acting on a scheme \(X\). A morphism
\(\phi:W\to X\) gives a correspondence
\(\Gamma_W\to G\times X\), obtained by base change from the action
map. We record the elementary base-change facts for this
correspondence; they will be used in Proposition~\ref{prop:key-proposition}
and Lemma~\ref{lem:structure-U_B}.

\begin{nota}\label{nota:group-multi-action}
\begin{enumerate} [(a)]
    \item Let \(m:G\times G\to G\) be multiplication, and let
    \(\mu:G\times X\to X\) be the left action of \(G\) on \(X\).
    \item For a closed point \(x\in X\), let
    \(\mu_x:G\to X\) be the orbit map, i.e. the restriction of \(\mu\)
    to \(G\times\{x\}\).
    \item For a scheme \(Y\), write \(1_Y\) for the identity of \(Y\).
    \item The multiplication map fits into the commutative diagram

\begin{equation}
\label{eqn:grp-multiplication}
\begin{tikzcd}
G \times G \ar[dr,"\text{pr}_{1}"'] \ar[rr,"\simeq","\tau"']& & G \times G \ar[dl,"m"]\\
& G &. 
\end{tikzcd}
\end{equation}

Here \(\tau:G\times G\to G\times G\) is given by
\(\tau(g,h)=(gh^{-1},h)\). Thus \(\tau\) is an isomorphism over \(G\),
where \(G\times G\) is viewed over \(G\) through the second projection.

\end{enumerate}
\end{nota}

\begin{lem}\label{lem:base-change-multiplication}
Let \(\phi:W\to G\) be a morphism, and let
\(\psi:\Gamma_W\to G\times G\) be the base change of \(\phi\) along
the multiplication map \(m:G\times G\to G\). Then there is a cartesian
diagram

\[
\begin{tikzcd}
\Gamma_{W} \ar[r,"\psi"]  & G \times G \\
W \times G \ar[u,"\prescript{}{W}{\tau}"',"\wr"]  \ar[r,"\phi \times 1_{G}"'] & G \times G \ar[u,"\tau"',"\wr"].
\end{tikzcd}
\]

\end{lem}

\begin{proof}
The result follows immediately from (\ref{eqn:grp-multiplication}).
\end{proof}

For a morphism \(f:Z\to G\), base change of \(\tau\) along
\(1_G\times f:G\times Z\to G\times G\) similarly produces an
isomorphism \(\tau_Z:G\times Z\simeq G\times Z\) over \(G\).

We use two standard facts about \'{e}tale cohomology.

\begin{lem}\label{lem:etale-vanishing-smooth-surjective}
Let \(\pi:Y\to X\) be smooth and surjective, and let
\(\phi:K_1\to K_2\) be a morphism of sheaves on \(X\). Then \(\phi\)
is an isomorphism if and only if \(\pi^*\phi\) is an isomorphism.
\end{lem}

\begin{proof}
Taking the cone of \(\phi\), it suffices to prove that \(K=0\) if and
only if \(\pi^*K=0\). Since \(\pi\) is smooth and surjective, it has
sections \'{e}tale locally on \(X\) \cite[Corollary 17.16.3
(ii)]{egaIV4}; the assertion follows.
\end{proof}

\begin{lem}\label{lem:kunneth-formula-projection}
  Let \(f_1\times f_2:Y_1\times Y_2\to X_1\times X_2\) be a product
  morphism. For sheaves \(K_1\) on \(Y_1\) and \(K_2\) on \(Y_2\),
  there are natural isomorphisms
  \[
(f_{1} \times f_{2})_{?}(K_{1} \boxtimes K_{2}) \simeq f_{1?}K_{1} \boxtimes f_{2?}K_{2},  
\]
\noindent for \(? \in \{!,*\}\).
\end{lem}

\begin{proof}
For \(?=!\), this is \cite[XVII, 5.4.3]{ArtinGrothendieckSGA4}.
The case \(?=*\) follows by biduality from
\cite[Th\'eor\`emes de Finitude, 4.3]{SGA4.5}.
\end{proof}

\subsection{The universal family and base change}

Let \(X\) be equipped with a left action of an algebraic group \(G\),
with action map \(\mu:G\times X\to X\). Let \(\phi:W\to X\) be a
morphism, and let \(\psi:\Gamma_W\to G\times X\) be the base change of
\(\phi\) along \(\mu\). For a morphism \(f:Z\to X\), write
\(f_G=1_G\times f:G\times Z\to G\times X\). The morphisms
\(\psi^Z\) and \(f^Z_G\) are defined by the cartesian diagram
\begin{equation}
\label{eqn:key-cartesian-diagram}
\begin{tikzcd}
\Gamma^{Z}_{W} \ar[r,"\psi^{Z}"] \ar[d,"f_{G}^{Z}"] & G \times Z \ar[d,"f_{G}"] \\
\Gamma_{W} \ar[r,"\psi"] & G \times X
\end{tikzcd}
\end{equation}

The proof rests on one observation. After base change along an orbit
map \(\mu_x\), the action becomes the left translation action of \(G\)
on itself. In that case the universal correspondence is a product, and
the desired base-change isomorphism is the K\"unneth formula.

\begin{prop}[Base change for the universal family]\label{prop:key-proposition}
Suppose that, for some (and hence every) closed point \(x\in X\), the
orbit map \(\mu_x\) is smooth and surjective. Let \(K\) be a sheaf on
\(Z\), and let \(K_{\Gamma}\) be its pullback to \(\Gamma^Z_W\). Then
the natural map
\[
\cT: \psi_{!}f^{Z}_{G*}K_{\Gamma} \to f_{G*}\psi^{Z}_{!}K_{\Gamma},
\]
is an isomorphism.
\end{prop}

\begin{proof}
  By Lemma~\ref{lem:etale-vanishing-smooth-surjective}, it is enough
  to check the map after the smooth surjective pullback
  \(1_G\times\mu_x:G\times G\to G\times X\). Replacing
  \(W,Z,\phi,f\) by their base changes along \(\mu_x\), we may assume
  \(X=G\) and that \(G\) acts on itself by left translation.
  Lemma~\ref{lem:base-change-multiplication} and base change of \(\tau\)
  along \(1_G\times f\) give the following diagram with cartesian
  squares:

  \begin{equation}\label{eq:key-cartesian-diagram-1}
    \begin{tikzcd}
      W \times Z \ar[r, "\phi \times 1_{Z}"] \ar[d,"\wr"] \ar[ddd, bend right=60, "1_{W} \times f"'] & G \times Z \ar[d,"\tau_{Z}","\wr"'] \ar[ddd, bend left=60, "1_{G} \times f"] \\
      \Gamma^{Z}_{W} \ar[r,"\psi^{Z}"] \ar[d,"f^{Z}_{G}"'] & G \times Z \ar[d,"f_{G}"] \\
      \Gamma_{W} \ar[r,"\psi"] & G \times G \\
      W \times G \ar[u,"\lsub{W}{\tau}","\wr"'] \ar[r,"\phi \times 1_{G}"] & G \times G \ar[u,"\tau"',"\wr"].
    \end{tikzcd}  
  \end{equation}

Under the identifications
\(\Gamma^Z_W\simeq W\times Z\) and \(\Gamma_W\simeq W\times G\),
the morphisms \(\psi^Z,\psi,f^Z_G,f_G\) become respectively
\(\phi\times 1_Z\), \(\phi\times 1_G\), \(1_W\times f\), and
\(1_G\times f\), and \(K_\Gamma\) becomes \(\Lambda\boxtimes K\).
With these source identifications and the target identifications given
by \(\tau_Z\) and \(\tau\), the natural transformation \(\cT\) is the
K\"unneth comparison map. Lemma~\ref{lem:kunneth-formula-projection}
therefore gives the assertion.
 
\end{proof}

\begin{rem}\label{rem:comparison-base-change-theorem}
Analogues of Proposition~\ref{prop:key-proposition} are central in
several proofs of generic vanishing; see \cite[Lemma 3.3]{Bei87},
\cite[p.\ 2101]{deCM10}, and
\cite[Proposition 1.7]{schurmann2025new}. The point here is that the
base-change isomorphism is universal. It uses neither characteristic
zero nor smoothness of \(W\) and \(Z\), and it imposes no genericity
condition.
\end{rem}

\section{A Radon transform from $X$ to $G$}
\label{sec:affine=smooth-correspondence}

We now attach to a triple \((G,X,\phi:W\to X)\), where \(G\) acts on
\(X\), a Radon transform from sheaves on \(X\) to sheaves on \(G\). In
the incidence correspondence of Example~\ref{ex:incidence-correspondence-example},
this is the pullback of the usual Radon transform \cite{Br}.


\subsection{The Radon transform}

Let \(G\) act on \(X\), let \(\phi:W\to X\) be a morphism, and let
\(\psi:\Gamma_W\to G\times X\) be the base change of \(\phi\) along
the action map \(\mu:G\times X\to X\). Let
\(\pi_W^\vee:\Gamma_W\to G\) and \(\pi_W:\Gamma_W\to X\) be the
induced projections.

\begin{equation}
\label{eqn:basic-diagram}
\begin{tikzcd}[row sep=large, column sep=large]
\Gamma_W \ar[r, "\psi"] 
         \ar[dr, bend right=20, "\pi^{\vee}_{W}"'] 
         \ar[drr, bend left=50, "\pi_W"]
  & G \times X \ar[d, "\mathrm{pr}_G"'] \ar[dr, "\mathrm{pr}_X"] \\
& G & X  & W \ar[l,"\phi"]
\end{tikzcd}
\end{equation}

\begin{lem}\label{lem:structure-U_B}
With this notation:
\begin{enumerate}[(a)]
\item If \(W\) is affine, then \(\pi_W^\vee\) is affine.
\item Suppose that \(W\) is smooth over \(k\), and that, for some (and
hence every) closed point \(x\in X\), the orbit map
\(\mu_x:G\to X\) is smooth and surjective. Then \(\pi_W\) is smooth.
\end{enumerate}
\end{lem}

\begin{proof}
Consider the commutative diagram
\begin{equation}
\label{eqn:action-projection}
\begin{tikzcd}
G \times X \ar[rr,"\sigma"',"\simeq"] \ar[dr,"\text{pr}_{X}"'] &  & G \times X \ar[dl,"\mu"]\\
& X &
\end{tikzcd}
\end{equation}
where \(\sigma(g,x)=(g,g^{-1}x)\), and \(\mathrm{pr}_X\) is projection
to the second factor. Base change along \(\phi:W\to X\) gives an
isomorphism \(\sigma':G\times W\simeq \Gamma_W\). Since
\(\mathrm{pr}_G\circ\sigma^{-1}=\mathrm{pr}_G\), we have
\[
\pi^\vee_W=\mathrm{pr}_G\circ(\sigma')^{-1}.
\]
Thus \(\pi_W^\vee\) identifies with the projection \(G\times W\to G\),
which is affine when \(W\) is affine.

For (b), smoothness descends under faithfully flat morphisms
\cite[\href{https://stacks.math.columbia.edu/tag/02VL}{Lemma
02VL}]{stacks-project}. We may therefore test \(\pi_W\) after base
change along the smooth surjection \(\mu_x:G\to X\). Let
\(\tilde W_x=W\times_X G\), and let
\(\tilde\phi:\tilde W_x\to G\) be the induced morphism. Since \(W\)
and \(\mu_x\) are smooth, \(\tilde W_x\) is smooth over \(k\). The
pullback \(\Gamma_W\times_X G\) is also the pullback of
\(\psi:\Gamma_W\to G\times X\) along
\(1_G\times\mu_x:G\times G\to G\times X\). The identity
\(\mu\circ(1_G\times\mu_x)=\mu_x\circ m\) identifies it with the base
change of \(\tilde\phi\) along \(m:G\times G\to G\). By
Lemma~\ref{lem:base-change-multiplication}, this is
\(\tilde W_x\times G\), and \(\pi_W\) becomes projection to the second
factor. This projection is smooth.
\end{proof}

We now assume that
\begin{enumerate}[(i)]
\item \(G\), \(X\), and \(W\) are irreducible;
\item for a closed point \(x\in X\), the orbit map
\(\mu_x:G\to X\) is smooth and surjective;
\item \(W\) is smooth over \(k\), so that \(\pi_W\) is smooth by
Lemma~\ref{lem:structure-U_B}(b).
\end{enumerate}

\begin{nota}
For a smooth morphism \(f\) of relative dimension \(d\), write
\(f^\dagger=f^*[d]=f^![-d]\). We suppress Tate twists, which play no
role in the arguments below.
\end{nota}

\begin{defn}\label{defn:radon-transform}
The \emph{Radon transform} associated with \((G,X,\phi:W\to X)\) is
the functor
\(\sR_{W!}:D^b_c(X)\to D^b_c(G)\) defined by
\begin{equation}
\label{eqn:radon-transform-defn-!}
\sR_{W!}:=\mathrm{pr}_{G*}\psi_{!}\pi^{\dagger}_{W}
\end{equation}
\end{defn}

 \begin{example}\label{ex:incidence-correspondence-example}
 Let $V/k$ be a finite-dimensional vector space, and denote by $X$ the projective space $\mathbb{P}(V)$ parametrizing one-dimensional linear subspaces of $V$. Let $X^{\vee}$ denote the dual projective space $\mathbb{P}(V^{\vee})$, parametrizing codimension-$1$ linear subspaces of $V$. Let $\sH \subseteq X^{\vee} \times X$ be the universal hyperplane section. Then $\sH \to X$ is a projective bundle and hence Zariski locally (on $X$) a product.

 Fix a hyperplane $\phi: W \to X$. The orbit map $\mu_{W}: \mathrm{GL}(V) \to X^{\vee}$ is then given by the orbit of $W$ under the right action of $\mathrm{GL}(V)$ on $X^{\vee}$. The map $\mu_{W}$ is smooth and surjective, and if we base change $\sH \subseteq X^{\vee} \times X$ along $\mu_{W} \times 1_{X}$, we obtain a correspondence\footnote{For us a \textit{correspondence} from a scheme $Y$ to a scheme $X$ is simply a morphism $\Gamma \to Y \times X$.} from $\mathrm{GL}(V)$ to $X$ which, in the notation of Proposition \ref{prop:key-proposition}, is $\Gamma_{W}$. The Radon transform in (\ref{eqn:radon-transform-defn-!}) is then the pullback of Brylinski's Radon transform \cite{Br} along the smooth and surjective morphism $\mu_{W}: \mathrm{GL}(V) \to X^{\vee}$.

 \end{example}

Let \(f:Z\to X\) be any morphism. We shall use the following diagram,
whose parallelograms are cartesian.

\begin{equation}
\label{eqn:diagram-with-all-actors}
\begin{tikzcd}[row sep=large, column sep=large]
  \Gamma_{W}^{Z}
    \arrow[d,"f_{G}^{Z}"']
    \arrow[r,"\psi^{Z}"]
    \arrow[drr, dotted, bend left=85, "\pi^{Z}_W"']
  & G\times Z
    \arrow[d,"f_{G}"']
    \arrow[dr,"\mathrm{pr}_{Z}"] 
  & \\[-2pt]
  \Gamma_{W}
    \arrow[r,"\psi"]
    \arrow[d, dotted, "\pi_{W}^{\vee}"']
    \arrow[drr, dotted, bend right=35, "\pi_W"']
  & G\times X
    \arrow[dl,"\mathrm{pr}_{G}"']
    \arrow[dr,"\mathrm{pr}_{X}"']
  & Z \arrow[d,"f"] 
  & \\[-2pt]
  G & & X & W \arrow[l,"\phi"']
\end{tikzcd}
\end{equation}

\begin{lem}\label{lem:t-exactness-lemma}
The natural transformation of functors
\[
\mathrm{pr}_{G*}f_{G*}\psi^{Z}_{!}\pi_{W}^{Z\dagger} \to\sR_{W!}\circ f_{*}, 
\]
\noindent from \(\mathrm{D}^{b}_{c}(Z)\) to \(\mathrm{D}^{b}_{c}(G)\)
is an isomorphism.
\end{lem}

\begin{proof}
This is Proposition~\ref{prop:key-proposition}, followed by smooth
base change for the cartesian diagram
\[
\begin{tikzcd}
\Gamma_{W}^{Z} \ar[r,"\pi_{W}^{Z}"] \ar[d,"f^{Z}_{G}"] & Z \ar[d,"f"] \\
\Gamma_{W} \ar[r,"\pi_{W}"] & X. 
\end{tikzcd}
\]
\end{proof}

\subsection{Hypotheses for $t$-exactness of $\sR_{W!}\circ f_{*}$}
\label{sec:set-up-t-exactness}

We now impose the hypotheses under which the Radon transform is
\(t\)-exact.

\begin{enumerate}[(i)]
\item \(X\) is proper and connected over \(k\), and carries an action
of a connected algebraic group \(G\).
\item For a closed point \(x\in X\), the orbit map \(\mu_x:G\to X\)
is smooth and surjective.\footnote{The argument only uses the existence
of \'{e}tale-local sections of \(\mu_x\), which follows from smooth
surjectivity; see \cite[Corollaire 17.16.3]{egaIV4}.}
\item The morphism \(\phi:W\to X\) has \(W\) smooth, connected, and
affine over \(k\).\footnote{In Example~\ref{ex:incidence-correspondence-example},
\(W\) is the chosen generator of the incidence correspondence.} By
Lemma~\ref{lem:structure-U_B}, the morphisms \(\pi_W\) and
\(\pi_W^\vee\) in \eqref{eqn:basic-diagram} are respectively smooth
and affine. Thus \(\pi_W^\dagger\) is \(t\)-exact
\cite[Proposition 4.2.5]{BBDG18}, and \(\pi^\vee_{W!}\) is left
\(t\)-exact by Artin vanishing \cite[Th\'eor\`eme 4.1.1]{BBDG18}.
\item The morphism \(\phi\) is quasi-finite. Since \(W\) is affine and
\(X\) is separated, \(\phi\) is affine. Hence its base changes
\(\psi\) and \(\psi^Z\) are quasi-finite and affine; by
\cite[Corollaire 4.1.3]{BBDG18}, \(\psi^Z_!\) is \(t\)-exact.
\item The morphism \(f:Z\to X\) is quasi-finite, with \(Z\) affine. In
particular \(f_!\) and \(f_*\) are \(t\)-exact for the perverse
\(t\)-structure \cite[Corollaire 4.1.3]{BBDG18}.
\item The morphism \(\pi^Z_W\) is the base change of the smooth
morphism \(\pi_W\), of constant relative dimension. Hence
\(\pi_W^{Z\dagger}\) is \(t\)-exact.
\end{enumerate}

\begin{thm}[$t$-exactness]\label{thm:radon-t-exact}
Under the assumptions above, the functor $\sR_{W!}\circ f_{*}$ is $t$-exact for the perverse $t$-structures on $\mathrm{D}^{b}_{c}(Z)$ and $\mathrm{D}^{b}_{c}(G)$.
\end{thm}

\begin{proof}

We prove left and right \(t\)-exactness separately.

\textbf{Left $t$-exactness of $\sR_{W!} \circ f_{*}$:}
By (v), \(f_*\) is \(t\)-exact. Since \(X\) is proper,
\[
\sR_{W!}=\pi^{\vee}_{W!}\pi_{W}^{\dagger}.
\]
The left \(t\)-exactness of \(\sR_{W!}\circ f_*\) follows from (iii).
\newline
\textbf{Right $t$-exactness of $\sR_{W!} \circ f_{*}$:}
By Lemma~\ref{lem:t-exactness-lemma}, it suffices to prove right
\(t\)-exactness of
\(\mathrm{pr}_{G*}f_{G*}\psi^{Z}_{!}\pi_{W}^{Z\dagger}\). The functors
\(\psi^Z_!\) and \(\pi_W^{Z\dagger}\) are \(t\)-exact by (iv) and
(vi). The composite \(\mathrm{pr}_{G*}f_{G*}\) is the direct image for
the projection \(G\times Z\to G\), which is affine because \(Z\) is
affine. Artin vanishing gives the required right \(t\)-exactness.

\end{proof}

\begin{nota}\label{nota:generic-vanishing-scheme}
For the application of Theorem~\ref{thm:radon-t-exact}, we use the
following notation.
\begin{enumerate}[(a)]

\item For \(g\in G(k)\), define \(gW\) by the cartesian diagram
\[
\begin{tikzcd}
X \ar[d,"g^{-1}"] & gW \ar[l,"_{g}\phi"] \ar[d] \\
X & W \ar[l,"\phi"]
\end{tikzcd}
\]

\item Define \(gW\cap_{f,\phi}Z\) by the cartesian diagram
\begin{equation}
\label{dia:intersection-given-g}
\begin{tikzcd}
Z \ar[d,"f"] & gW \cap_{f,\phi}Z \ar[d,"_{g}f^{Z}"] \ar[l,"_{g}\phi^{Z}"] \\
X & gW \ar[l,"~_{g}\phi"]
\end{tikzcd}
\end{equation} 

\item For a sheaf \(K\) on \(Z\), set
\[
G_{i}(K):=\{g \in G(k)\mid
\mathrm{H}_{c}^{i}(gW,{}_{g}f^{Z}_{*}{}_{g}\phi^{Z*}K
[\dim W-\dim X]) \neq 0\}.
\]
\end{enumerate}
\end{nota}

The main application is the following generic vanishing theorem.

\begin{cor}\label{cor:generic-vanishing-general}
Let \(K\) be a perverse sheaf on \(Z\). Then there exists a dense open
subset \(U\subset G\) such that
\begin{enumerate}[(i)]
\item \(U\cap G_i(K)=\emptyset\) for \(i<0\);
\item for \(i\geq 0\), \(\dim(U\cap G_i(K))\leq \dim G-i\).
\end{enumerate}

In particular, for \(g\) in a Zariski dense open subset of \(G(k)\) and
for \(i\neq \dim W-\dim X\),
\begin{equation}
\label{eqn:main-vanishing}
\mathrm{H}^{i}_{c}(gW,~_{g}f^{Z}_{*}~_{g}\phi^{Z*}K) \simeq \mathrm{H}^{i}(Z,~_{g}\phi^{Z}_{!}~_{g}\phi^{Z*}K) = 0.
\end{equation}
\end{cor}

\begin{proof}
By Theorem~\ref{thm:radon-t-exact},
\(\sR_{W!}\circ f_{*}(K)\) is perverse on \(G\). Hence
\[
\mathcal{H}^{a}(\sR_{W!}\circ f_{*}(K)) = 0 \quad\text{for}\quad a > 0\ \text{or}\ a < -\dim(G),
\]
and
\[
\dim(\mathrm{Supp}\,\mathcal{H}^{a}(\sR_{W!}\circ f_{*}(K))) \leq -a \quad\text{for}\quad -\dim(G) \leq a \leq 0.
\]
By Deligne's generic base change theorem
\cite[Chapitre 7, Th\'eor\`eme 1.9]{SGA4.5}, there exists a dense open
subset \(U\subset G\) such that, for all \(g\in U(k)\),\footnote{The
shift by \(\dim G+\dim W-\dim X\) is the relative dimension of
\(\pi_W\).}
\[
(\sR_{W!}\circ f_{*}(K))_{g} \simeq R\Gamma_{c}(gW,~_{g}f^{Z}_{*}~_{g}\phi^{Z*}K[\dim(G)+\dim(W)-\dim(X)]).
\]
Taking $\mathcal{H}^{i-\dim(G)}$ of both sides and comparing with the definition of $G_{i}(K)$ in Notation \ref{nota:generic-vanishing-scheme}, we have
\[
g \in G_{i}(K) \cap U \iff \mathcal{H}^{i-\dim(G)}(\sR_{W!}\circ f_{*}(K))_{g} \neq 0.
\]
Parts (i) and (ii) follow from these perversity bounds, applied with
\(a=i-\dim G\).

For the final statement, Lemma~\ref{lem:equality-cohomology-with-without}
gives a dense open subset \(U_1\subset G\) such that
\[
\mathrm{H}^{i}_{c}(gW,~_{g}f^{Z}_{*}~_{g}\phi^{Z*}K) \simeq \mathrm{H}^{i}(Z,~_{g}\phi^{Z}_{!}~_{g}\phi^{Z*}K)
\]
for all \(g\in U_1(k)\) and all \(i\). Part (i) gives vanishing for
\(i<\dim W-\dim X\) and \(g\in U\). By part (ii), for each
\(j\geq 1\), the subset
\(G_{j+\dim W-\dim X}(K)\cap U\) has dimension at most \(\dim G-j\),
and is therefore proper in \(U\). Only finitely many such subsets are
non-empty, since \(\mathrm{H}^{i}_{c}(gW,\cdot)=0\) for \(|i|\gg 0\).
Removing their union from \(U\cap U_1\) gives a dense open subset on
which all groups with \(i\neq \dim W-\dim X\) vanish.
\end{proof}

\begin{lem}\label{lem:equality-cohomology-with-without}
For any sheaf \(K\) on \(\Gamma^Z_W\), there is a dense open subset
\(U\subset G\) such that, for every \(g\in U(k)\), there is a natural
isomorphism
\[
R\Gamma_{c}(gW,~_{g}f^{Z}_{*}K) \simeq R\Gamma(Z,~_{g}\phi^{Z}_{!}K).
\]
\end{lem}

\begin{proof}
Proposition~\ref{prop:key-proposition} and Deligne's generic base
change theorem \cite[Chapitre 7, Th\'eor\`eme 1.9 ]{SGA4.5} give a
dense open \(U\subset G\) such that, for all \(g\in U(k)\),
\begin{equation}
\label{eqn:equality-cohomology-with-without-1}
~_{g}\phi_{!}~_{g}f^{Z}_{*}K \simeq f_{*}~_{g}\phi^{Z}_{!}K.
\end{equation}
Since \(X\) is proper over \(k\),
\begin{equation}
\label{eqn:equality-cohomology-with-without-2}
R\Gamma_{c}(gW,~_{g}f^{Z}_{*}K) \simeq R\Gamma(X,~_{g}\phi_{!}~_{g}f_{*}^{Z}K).
\end{equation}
Combining \eqref{eqn:equality-cohomology-with-without-1} and
\eqref{eqn:equality-cohomology-with-without-2} gives the result.
\end{proof}

We conclude with the sign of the Euler characteristic.

\begin{cor}\label{cor:positivity-euler-char}
Assume that the coefficients form a field, and let \(K\) be a perverse
sheaf on \(Z\). There exists a dense open subset \(U\subset G\) such
that, for every \(g\in U(k)\),
\[
(-1)^{\dim X-\dim W}\chi(gW \cap_{f,\phi}Z,{}_{g}\phi^{Z*}K) \geq 0.
\]
\end{cor}

\begin{proof}
By Corollary~\ref{cor:generic-vanishing-general},
\[
(-1)^{\dim(X)-\dim(W)}\chi(Z,~_{g}\phi^{Z}_{!}~_{g}\phi^{Z*}K) \geq 0.
\]
Laumon's comparison \cite{Lau81} gives
\[
\chi(Z,~_{g}\phi^{Z}_{!}~_{g}\phi^{Z*}K) = \chi_{c}(Z,~_{g}\phi^{Z}_{!}~_{g}\phi^{Z*}K).
\]
The identity
\[
R\Gamma_{c}(gW \cap_{f,\phi} Z,{}_{g}\phi^{Z*}K)
\simeq
R\Gamma_{c}(Z,{}_{g}\phi^{Z}_{!}{}_{g}\phi^{Z*}K)
\]
and applying the same comparison again yield
\[
\chi_{c}(Z,~_{g}\phi^{Z}_{!}~_{g}\phi^{Z*}K) = \chi(gW \cap_{f,\phi} Z,~_{g}\phi^{Z*}K),
\]
and the corollary follows.
\end{proof}

\section{Arithmetic corollaries}
\label{sec:arithmetic-corollaries}

In this section we specialize to the case in which the base field is
the algebraic closure of a finite field, all the data are defined
over a common finite subfield, and \(K_0\) is a mixed perverse sheaf
in the sense of \cite[5.1.5]{BBDG18}. The structural results of
Section~\ref{sec:set-up-t-exactness} then admit useful arithmetic
refinements: the trace function of the Radon transform
\(\sR_{W!}\circ f_{*}(K_0)\) is, on a dense open subset of \(G\), the
Frobenius trace on a single cohomology group
(Corollary~\ref{cor:trace-arith}). Under some assumptions on \(K_0\)
at infinity, this trace identity can be bounded using standard
Weil~II apparatus (Corollary~\ref{cor:weil-bound-clean-arith}).

Throughout this section we fix a finite field \(k_0\) of cardinality
\(q\) and characteristic \(p\), an algebraic closure \(k\) of
\(k_0\), and a prime \(\ell\neq p\). We work with coefficients in
\(\overline{\bbQ}_{\ell}\) and use the language of mixed sheaves on
\(k_0\)-schemes from \cite[\S 5.1]{BBDG18}. We assume that the data
\((G, X, W, Z, \phi, f)\) of Section~\ref{sec:set-up-t-exactness}
descend to \(k_0\): we fix \(k_0\)-forms \(G_0, X_0, W_0, Z_0\), a
\(k_0\)-action \(\mu_0:G_0\times X_0\to X_0\), and
\(k_0\)-morphisms \(\phi_0:W_0\to X_0\), \(f_0:Z_0\to X_0\) whose
base change to \(k\) recovers the action and the data
\((G, X, W, Z, \phi, f)\). Let \(K_0\) be a mixed perverse sheaf on
\(Z_0\); we write \(K\) for its pullback to \(Z\).

Throughout the section we abbreviate
\[
a:=\dim G,\qquad x:=\dim X,\qquad b:=\dim W,\qquad
r:=a+b-x,\qquad \delta:=b-x;
\]
in particular \(r\) is the relative dimension of the smooth morphism
\(\pi_W:\Gamma_W\to X\) (Lemma~\ref{lem:structure-U_B}), and
\(r-a=\delta\). As in the rest of the paper,
\(\pi_W^\dagger\) means the untwisted functor \(\pi_W^*[r]\). Thus
the trace identities below are identities for the unnormalised
cohomology groups that occur in the exponential sums.

By Theorem~\ref{thm:radon-t-exact}, the complex
\(\sR_{W!}\circ f_{*}(K)\) is a perverse sheaf on \(G\); since the
construction descends to \(k_0\), it underlies a mixed perverse
sheaf on \(G_0\), which we denote \(\sR_{W!}\circ f_{*}(K_0)\). Its
standard cohomology sheaves
\(\mathcal{H}^{m}(\sR_{W!}\circ f_{*}(K_0))\) are constructible on
\(G_0\), with \(k_0\)-closed supports in \(G_0\).

\subsection{Descent of the good open}
\label{sec:descent-U-arith}

\begin{lem}\label{lem:U-descends-arith}
There exist dense open subsets \(U^\circ_0\subset U_0\subset G_0\)
over \(k_0\), with pullbacks \(U^\circ\subset U\subset G\), satisfying
the following properties; the isomorphisms below underlie morphisms
of mixed complexes over \(U_0\), and hence specialize to
\(\Frob_{g,n}\)-equivariant isomorphisms at every
\(g\in U(\bbF_{q^n})\).
\begin{enumerate}[(i)]
\item For every \(g\in U(k)\), generic base change yields
\[
(\sR_{W!}\circ f_*K)_{\bar g}
   \;\simeq\;
R\Gamma_c\bigl(gW,\,{}_{g}f^{Z}_{*}\,{}_{g}\phi^{Z*}K[r]\bigr).
\]
\item For every \(g\in U(k)\),
Lemma~\ref{lem:equality-cohomology-with-without} gives
\[
R\Gamma_c\bigl(gW,\,{}_{g}f^{Z}_{*}\,{}_{g}\phi^{Z*}K\bigr)
   \;\simeq\;
R\Gamma\bigl(Z,\,{}_{g}\phi^{Z}_{!}\,{}_{g}\phi^{Z*}K\bigr).
\]
\item For every \(g\in U^\circ(k)\),
\(\mathcal H^m(\sR_{W!}\circ f_*K)_{\bar g}=0\) for all \(m\neq -a\).
\end{enumerate}
Concretely, after choosing \(U_0\) to satisfy (i) and (ii), one may
take
\[
U^{\circ}_0 := U_0 \;\setminus\; \overline{\bigcup_{i \geq 1}\bigl(G_{i}(K_0)\cap U_0\bigr)},
\]
where \(G_i(K_0)\subset G_0\) is the \(k_0\)-form of \(G_i(K)\) given
by the support of \(\mathcal H^{i-a}(\sR_{W!}\circ f_*K_0)|_{U_0}\).
This is dense because Corollary~\ref{cor:generic-vanishing-general}(ii)
gives \(\dim(G_i(K)\cap U)\leq a-i<a\) for every \(i\geq 1\).
\end{lem}

\begin{proof}
Apply Deligne's generic base change theorem
\cite[Ch.~7, Th\'eor\`eme~1.9]{SGA4.5}, over \(k_0\), to the
canonical base-change morphism giving the displayed stalk isomorphism
for the mixed sheaf \(\sR_{W!}\circ f_*K_0\). Removing the support of
the cone of this morphism gives a dense open \(U^{(1)}_0\subset G_0\)
on which (i) holds. The same construction applied to the canonical
morphism underlying Lemma~\ref{lem:equality-cohomology-with-without}
gives a dense open \(U^{(2)}_0\subset G_0\) on which (ii) holds. We
take \(U_0=U^{(1)}_0\cap U^{(2)}_0\).

On \(U\), generic base change identifies \(G_i(K)\cap U\) with the
support, inside \(U\), of the constructible sheaf
\(\mathcal H^{i-a}(\sR_{W!}\circ f_*K)|_U\). The latter sheaf
underlies a constructible sheaf on \(U_0\), so its support descends
to a \(k_0\)-constructible subset \(G_i(K_0)\cap U_0\) of \(U_0\)
with \(k_0\)-closed closure. Only finitely many of the
\(\mathcal H^{m}(\sR_{W!}\circ f_*K)\) are non-zero, and by
Corollary~\ref{cor:generic-vanishing-general}(i) the loci with
\(i<0\) are empty on \(U\). Removing from \(U_0\) the closures of
those \(G_i(K_0)\cap U_0\) with \(i\geq 1\) gives an open
\(U^{\circ}_0\) over \(k_0\) on which the only standard cohomology
sheaf of \(\sR_{W!}\circ f_*K_0\) that may be non-zero is
\(\mathcal H^{-a}\), giving (iii).
\end{proof}

We fix such \(k_0\)-models \(U_0, U^{\circ}_0\) (and their pullbacks
\(U, U^{\circ}\)) for the remainder of the section.

\subsection{The trace-function identity}
\label{sec:trace-arith}

For a mixed complex \(\mathcal{F}_0\) on \(G_0\) with pullback
\(\mathcal F\) to \(G\), recall that the \emph{trace function} of
\(\mathcal{F}_0\) at an \(\bbF_{q^{n}}\)-rational point
\(g \in G_0(\bbF_{q^{n}})\) is
\[
t_{\mathcal{F}_0}(g; \bbF_{q^{n}})
   \;:=\; \sum_{m} (-1)^{m}\,\operatorname{Tr}\bigl(\Frob_{g,n} \,\big|\, \mathcal{H}^{m}(\mathcal{F})_{\bar{g}}\bigr),
\]
where \(\Frob_{g,n}\) is the geometric Frobenius acting on the stalk
at the geometric point \(\bar{g}\) above \(g\).

\begin{cor}[Trace-function identity]\label{cor:trace-arith}
Let \(K_0\) be a mixed perverse sheaf on \(Z_0\) and set
\(\mathcal{F}_0 := \sR_{W!}\circ f_{*}(K_0)\). For every
\(g \in U^{\circ}_0(\bbF_{q^{n}})\):
\begin{enumerate}[(a)]
\item The stalk \(\mathcal{F}_{\bar{g}}\) is concentrated in
cohomological degree \(-a\), and
\[
t_{\mathcal{F}_0}(g; \bbF_{q^{n}})
   = (-1)^{a}\,
     \operatorname{Tr}\Bigl(\Frob_{g,n}\,\Big|\,
       \mathrm{H}^{\delta}_{c}\bigl(gW,\;
         {}_{g}f^{Z}_{*}\,{}_{g}\phi^{Z*}K\bigr)\Bigr).
\]
\item The same trace can be computed on \(Z\):
\[
t_{\mathcal{F}_0}(g; \bbF_{q^{n}})
   = (-1)^{a}\,
     \operatorname{Tr}\Bigl(\Frob_{g,n}\,\Big|\,
       \mathrm{H}^{\delta}\bigl(Z,\;
         {}_{g}\phi^{Z}_{!}\,{}_{g}\phi^{Z*}K\bigr)\Bigr).
\]
\end{enumerate}
\end{cor}

\begin{proof}
Part (a). By Lemma~\ref{lem:U-descends-arith}(iii),
\(\mathcal H^m(\mathcal F)_{\bar g}=0\) for \(m\neq -a\), so the trace
sum collapses to
\[
t_{\mathcal{F}_0}(g; \bbF_{q^{n}})
   = (-1)^{-a}\,\operatorname{Tr}\bigl(\Frob_{g,n}\,\big|\,
       \mathcal{H}^{-a}(\mathcal{F})_{\bar{g}}\bigr).
\]
By Lemma~\ref{lem:U-descends-arith}(i),
\(\mathcal F_{\bar g}\simeq
R\Gamma_c\bigl(gW,\,{}_{g}f^{Z}_{*}\,{}_{g}\phi^{Z*}K[r]\bigr)\)
as a complex of \(\Frob_{g,n}\)-modules. Taking \(\mathcal H^{-a}\)
and using \(r-a=\delta\) gives
\[
\mathcal{H}^{-a}(\mathcal{F})_{\bar{g}}
   \;\simeq\; \mathrm{H}^{\delta}_{c}\bigl(gW,\;
       {}_{g}f^{Z}_{*}\,{}_{g}\phi^{Z*}K\bigr)
\]
as \(\Frob_{g,n}\)-modules. Since \((-1)^{-a}=(-1)^a\), part (a)
follows.

Part (b). Combine part (a) with the \(\Frob_{g,n}\)-equivariant
comparison of Lemma~\ref{lem:U-descends-arith}(ii) (after taking
\(\mathrm H^\delta\)).
\end{proof}

\subsection{A Lang--Weil estimate for the exceptional locus}
\label{sec:sparsity-arith}

\begin{cor}\label{cor:sparsity-arith}
There exist constants \(C_{i} \geq 0\) (for \(i \geq 1\)) and
\(C \geq 0\), depending on the data
\((G_0, X_0, W_0, Z_0, \phi_0, f_0, K_0)\) but independent of \(n\),
such that:
\begin{enumerate}[(i)]
\item for every \(i \geq 1\),
\[
\#\bigl((G_{i}(K_0) \cap U_0)(\bbF_{q^{n}})\bigr)
   \;\leq\; C_{i}\,(q^{n})^{a - i};
\]
\item the complement
\[
E(K_0) \;:=\; G_0 \setminus U^{\circ}_0 \;\subseteq\; G_0
\]
of the trace-function open \(U^{\circ}_0\) from
Lemma~\ref{lem:U-descends-arith} satisfies
\[
\#\bigl(E(K_0)(\bbF_{q^{n}})\bigr)
   \;\leq\; C\,(q^{n})^{a - 1}.
\]
\end{enumerate}
\end{cor}

\begin{proof}
(i) By Corollary~\ref{cor:generic-vanishing-general}(ii), the
intersection \(G_i(K)\cap U\) is constructible of dimension at most
\(a-i\). Its \(k_0\)-structure comes from
Lemma~\ref{lem:U-descends-arith}: on \(U_0\), generic base change
identifies it with the support of
\(\mathcal H^{i-a}(\sR_{W!}\circ f_*K_0)|_{U_0}\).

Stratify \(G_{i}(K_0) \cap U_0\) as a finite disjoint union of
locally closed \(k_0\)-subvarieties of \(G_0\), each of dimension at
most \(a-i\). The Lang--Weil estimate
\cite[Theorem~1]{LangWeil54}, applied to the closure of each stratum
\(S_{\alpha}\), yields
\[
\#S_{\alpha}(\bbF_{q^{n}})
   \;\leq\; \#\overline{S_{\alpha}}(\bbF_{q^{n}})
   \;\leq\; c_{\alpha}\,(q^{n})^{\dim S_{\alpha}}
   \;\leq\; c_{\alpha}\,(q^{n})^{a - i},
\]
with \(c_{\alpha}\) depending only on the geometry of
\(\overline{S_{\alpha}}\). Summing over the finitely many strata
gives the bound with \(C_{i} = \sum_{\alpha} c_{\alpha}\).

(ii) The complement \(E(K_0)=G_0 \setminus U^{\circ}_0\) is, by the
construction of \(U^{\circ}_0\) in Lemma~\ref{lem:U-descends-arith},
contained in
\((G_0 \setminus U_0) \cup \bigcup_{i \geq 1} \overline{G_{i}(K_0) \cap U_0}\).
The first piece \(G_0 \setminus U_0\) is a proper \(k_0\)-closed
subset of \(G_0\) (since \(U_0\) is dense open), hence of dimension
at most \(a-1\). The second piece is a finite union (only finitely
many \(G_{i}(K_0)\) are non-empty, since
\(\mathrm{H}^{i}_{c}(gW, \cdot)\) vanishes for \(|i|\) sufficiently
large) of \(k_0\)-closed subsets of dimension at most \(a-1\).
Lang--Weil applied directly to \(G_0\setminus U_0\) and to these
finitely many closures gives the stated bound.
\end{proof}

\begin{rem}\label{rem:sparsity-honest-arith}
For \(i \geq 2\), Corollary~\ref{cor:sparsity-arith}(i) gives the
sharper exponent \(a-i\) on the part of \(G_{i}(K_0)\) that lies
inside the good open \(U_0\). The intersection is unavoidable: a
priori \(G_{i}(K_0) \setminus U_0\) can fill all of the
codimension-\(1\) complement \(G_0 \setminus U_0\), in which case no
estimate sharper than \(O\bigl((q^{n})^{a-1}\bigr)\) is available for
the full \(G_{i}(K_0)\).
\end{rem}

\subsection{A Weil bound}
\label{sec:closing-arith}

We close by combining the trace identity of
Corollary~\ref{cor:trace-arith} with Deligne's Weil~II weight
bounds \cite{Del80} to obtain a uniform bound on the trace
function of \(\sR_{W!}\circ f_*K_0\). Beyond the standing
assumption that \(K_0\) is a mixed perverse sheaf, this requires
the forget-supports map \(Rf_{0!}K_0\to Rf_{0*}K_0\) to be an isomorphism (see
Remark~\ref{rem:clean-exponential-sums-arith} for the motivation).

\begin{cor}\label{cor:weil-bound-clean-arith}
Assume, in addition, that \(K_0\) is mixed of weights \(\leq w\), and
that the natural map
\[
Rf_{0!}K_0 \longrightarrow Rf_{0*}K_0
\]
is an isomorphism. Then for every \(g\in U^{\circ}_0(\bbF_{q^n})\),
\[
\bigl|\,t_{\sR_{W!}\circ f_{*}(K_0)}(g; \bbF_{q^{n}})\,\bigr|
   \;\leq\; N(g)\,(q^n)^{(w+\delta)/2},
\]
where
\[
N(g):=
\dim \mathrm{H}^{\delta}_{c}
   \bigl(gW,\,{}_{g}f^Z_*\,{}_{g}\phi^{Z*}K\bigr).
\]
Equivalently, by the generic vanishing theorem,
\[
N(g)
  = (-1)^{x-b}\,
     \chi\bigl(gW\cap_{f,\phi}Z,\,{}_{g}\phi^{Z*}K\bigr)
\]
is non-negative. After replacing \(U^{\circ}_0\) by a smaller dense
open, \(N(g)\) is locally constant in \(g\).
\end{cor}

\begin{proof}
The only point not contained in Corollary~\ref{cor:trace-arith} is
the weight bound. Since \(K_0\) is mixed of weights \(\leq w\),
Deligne's stability theorem for \(Rf_{0!}\)
\cite[3.3.1, 6.2.3]{Del80} gives that \(Rf_{0!}K_0\) is mixed of
weights \(\leq w\); the assumed isomorphism
\(Rf_{0!}K_0\simeq Rf_{0*}K_0\) transports this bound to
\(Rf_{0*}K_0\).

Recall the definition
\[
\sR_{W!}=\mathrm{pr}_{G*}\psi_!\pi_W^\dagger,
\qquad
\pi_W^\dagger = \pi_W^*[r],
\]
where \(r=a+b-x\) is the relative dimension of the smooth morphism
\(\pi_W:\Gamma_W\to X\). Smooth pullback \(\pi_W^*\) preserves the
weight bound on a complex, while the cohomological shift \([r]\)
raises the bound from \(\leq w\) to \(\leq w+r\). The functor
\(\psi_!\) preserves weights \(\leq w+r\)
(Weil~II \cite[3.3.1]{Del80}), and so does
\(\mathrm{pr}_{G*}=\mathrm{pr}_{G!}\) (proper, since \(X\) is proper).
Hence \(\sR_{W!}\circ f_{*}(K_0)\) is mixed of weights \(\leq w+r\)
as a complex on \(G_0\). By \cite[5.1.14(i)]{BBDG18}, its \((-a)\)-th
standard cohomology sheaf
\(\mathcal H^{-a}(\sR_{W!}\circ f_{*}(K_0))\) is mixed of weights
\(\leq (w+r)+(-a)=w+\delta\).

For \(g\in U^{\circ}_0(\bbF_{q^n})\), Corollary~\ref{cor:trace-arith}
identifies the stalk of this cohomology sheaf at \(\bar g\) with
\(\mathrm H^\delta_c(gW,\,{}_{g}f^Z_*\,{}_{g}\phi^{Z*}K)\). All
eigenvalues of \(\Frob_{g,n}\) on this group therefore have complex
absolute value at most \((q^n)^{(w+\delta)/2}\), and the displayed
estimate follows by summing the absolute values of the eigenvalues.
The equality with the signed Euler characteristic follows from the
concentration statement and
Corollary~\ref{cor:positivity-euler-char}.
\end{proof}

\begin{rem}\label{rem:clean-exponential-sums-arith}
By Zariski's main theorem, \(f_0\) factors as
\[
Z_0 \xhookrightarrow{j_0} \overline{Z_0} \xrightarrow{\overline{f_0}} X_0
\]
with \(j_0\) open and \(\overline{f_0}\) finite. Since \(\overline{f_0}\)
is finite, \(R\overline{f_{0,!}} = R\overline{f_{0,*}}\), so
\(Rj_{0,!}K_0\simeq Rj_{0,*}K_0\) on \(\overline{Z_0}\) suffices
for the hypothesis of Corollary~\ref{cor:weil-bound-clean-arith}.
Suppose furthermore that both \(j_{0,!}K_0\) and
\(j_{0,*}K_0\) are perverse on \(\overline{Z_0}\). The middle extension \(j_{0,!*}K_0\) is by
definition the image of the canonical map
\(j_{0,!}K_0\to j_{0,*}K_0\), so the hypothesis becomes
\[
j_{0,!}K_0\xrightarrow{\sim} j_{0,!*}K_0\xrightarrow{\sim} j_{0,*}K_0;
\]
i.e.\ \(K_0\) admits no boundary contribution along
\(\overline{Z_0}\setminus Z_0\), as in the motivating example
\cite[Example~2.2]{Kat93} of a Kummer sheaf on \(\mathbb A^1\).
\end{rem}

\bibliographystyle{alpha-custom}
\bibliography{references.bib}
\end{document}